\documentclass[12pt,reqno]{amsart}

\usepackage{amssymb,amsmath,euscript,enumerate,tikz}
\usepackage[margin=1in]{geometry} 
\usepackage{hyperref}
\usepackage{float}
\usepackage{xcolor}
\numberwithin{equation}{section}
\usepackage{caption}
\usepackage{tikz-cd}
\tikzstyle{vertex}=[circle,draw, inner sep=0pt, minimum size=1.5pt]

\usepackage{}

\newtheorem{theorem}{Theorem}[section]
\newtheorem{lemma}[theorem]{Lemma}
\newtheorem{proposition}[theorem]{Proposition}
\newtheorem{question}[theorem]{Question}

\newtheorem{corollary}[theorem]{Corollary}

\theoremstyle{definition}

\newtheorem{notation}[theorem]{Notation}
\newtheorem{example}[theorem]{Example}

\newtheorem{obs}[theorem]{Observation}
\theoremstyle{remark}
\newtheorem{remark}[theorem]{Remark}

\usepackage{color,soul}

\newcommand{\pd}{\operatorname{pd}}
\newcommand{\minh}{\operatorname{Minh}}

\newcommand{\supp}{\operatorname{supp}}
\newcommand{\ass}{\operatorname{Ass}}

\newcommand{\hgt}{\operatorname{ht}}
\newcommand{\Tor}{\mathrm{Tor}}
\newcommand{\reg}{\mathrm{reg}}

\begin{document}

	\title[Symbolic powers of certain cover ideals]{Symbolic powers of certain cover ideals of graphs}
	\author[Arvind Kumar]{Arvind Kumar}
	\email{arvkumar11@gmail.com}
	\address{Department of Mathematics, Indian Institute of Technology
        Madras, Chennai, INDIA - 600036}
	\author[R. Kumar]{Rajiv Kumar}
	\email{gargrajiv00@gmail.com }
	\address{Department of Mathematics
The LNM INSTITUTE OF Information Technology Jaipur, INDIA-302031}
	\author[Rajib Sarkar]{Rajib Sarkar}
	\email{rajib.sarkar63@gmail.com}
	\address{Department of Mathematics, Indian Institute of Technology
        Madras, Chennai, INDIA - 600036}
	\author[S. Selvaraja]{S Selvaraja}
	\email{selva.y2s@gmail.com, selvaraja@imsc.res.in}
	\address{The Institute of Mathematical Sciences, CIT campus, Taramani, Chennai,
INDIA - 600113}
	
	\begin{abstract}
		In this paper, 
		 we compute the regularity and Hilbert series of symbolic powers of cover ideal of a graph $G$ when $G$ is  either a crown graph or a complete multipartite graph.
		We also compute the multiplicity of symbolic powers of cover ideals in terms of 
the number of edges.
	\end{abstract}
	
	\thanks{AMS Classification 2010: 13D02, 13F20}

	\maketitle
\section{Introduction}
Symbolic powers of ideals have been studied intensely over the last two decades. 
We refer the reader to \cite{DDAGHN} for a review in this direction. 
There are many ideals  associated to graphs, for example edge ideals and cover ideals.
Let $G$ denote a finite simple (no loops, no multiple edges) undirected graph with the vertex set 
$V(G) =\{x_1,\ldots,x_n\}$ and edge set $E(G)$. For a graph $G$, by identifying the vertices with variables in 
$S=\mathbb{K}[x_1,\ldots,x_n]$, where $\mathbb{K}$ is a field, we associate squarefree monomial ideals, {\it edge ideal} $I(G)=\left(x_ix_j \mid \{x_i,x_j\} \in  E(G) \right) $ and {\it cover ideal} 
$J(G)= \left( \prod_{x \in w}x \mid w \text{ is a minimal vertex cover of } G\right)$.  By \cite[Proposition 2.7]{FHM}, $I(G)$ and $J(G)$ are dual to each other.
Recently, building a dictionary between combinatorial data of graphs and the algebraic properties of  corresponding ideals has 
been done by various authors
(cf. \cite{FHM}, \cite{HT19}, \cite{HTrung}, \cite{PhDT103J}, \cite{MMV19}, \cite{Seyed16}, \cite{Fakhari17}, \cite{Seyed},  \cite{villarreal_book}). 
In particular, establishing a relationship between Castelnuovo-Mumford regularity (or simply, regularity) 
of powers of ideals, Hilbert series of ideals and combinatorial invariants associated with graphs is an active area of 
research (cf. \cite{selvi_ha}, \cite{Good2013}, \cite{jayanthan}).

It was proved by Cutkosky,
Herzog and Trung \cite{CHT}, and independently Kodiyalam
\cite{vijay}, that if $I$ is a homogeneous ideal in $\mathbb{K}[x_1,
\ldots, x_n]$, then there exist non-negative integers $a, b$ and
$s_0$ such that $\reg(I^s) = as + b$ for all $s \geq s_0$. 
While the
coefficient $a$ is well-understood (cf. \cite{CHT}, \cite{Gu17}, \cite{vijay}, \cite{TW}), the free constant $b$ and  
stablization index $s_0=\min\{s \mid \reg(I^t)=at+b \text{ for } t\geq s\}$ are quite mysterious. In the case of symbolic powers,
Minh and Trung \cite{MT17}, ask the following question.
\begin{question}
	Let $I$ be a squarefree monomial ideal. Is $\reg(I^{(s)})$ a linear function for $s \gg 0$?
\end{question}
In \cite{HHT07},
Herzog, Hibi and Trung proved that, if $I$ is a monomial ideal, then $\reg(I^{(s)})$ is a quasi-linear
function for $s \gg 0$. For small dimension, more general results are
known in \cite{HHT02} and \cite{HT10}.
It is not known whether the regularity of symbolic powers of
squarefree monomial ideals is a linear function or not. In this article, 
we determine the  linear polynomial for the regularity of symbolic powers of certain cover ideals of graphs.

A crown graph $C_{n,n}$ is a graph obtained  from $K_{n.n}$ by
removing a perfect matching (see definition in Section \ref{pre}). 
The Betti numbers of edge ideal and representation number of crown graphs have been looked
by several authors \cite{GKP-crown}, \cite{Rather2018}.  
Since crown graph is a bipartite graph, by \cite[Corollary 2.6]{GRV05},
$J(G)^s=J(G)^{(s)}$ for all $s \geq 1$. 
In \cite{HTrung}, Hang and Trung proved that if $G$ is bipartite, then 
$b \leq |V(G)|-\deg(J(G))-1$ and $s_0 \leq |V(G)|+2$, where $\deg(J(G))=\max\{|C| : C \text{ is a minimal vertex cover  of }
G\}.$ In the case of crown graph we obtain that $b=0$ and $s_0=1$ (Theorem \ref{basecase}).

We then consider  complete multipartite graphs. The resolution of powers of cover ideals of
complete multipartite graphs and vanishing ideal of the parametrized algebraic toric set associated to 
complete multipartite graphs have already been studied by several authors \cite{JayanNeeraj}, \cite{RajivAjay}, 
\cite{RajivAjay18}, \cite{NV14}. 
We prove that, if $G$ is complete multipartite with partition
$V(G)=V_1 \cup \cdots \cup V_k$, then  $\reg(J(G)^{(s)})=s  \deg(J(G))+p-1$ for all $s \geq 1$, where 
$p=\min\{p_i : p_i=|V_i|\}$ (Theorem \ref{powerMultiThm}).

The Hilbert function, 
Hilbert series and Hilbert polynomial are important invariants in commutative algebra and algebraic geometry
that measure the growth of the dimension of its homogeneous components.
In general, computing   the Hilbert series  of $S/I$ is a 
difficult task when $I$ is a monomial ideal \cite{Bh1993}. 
In \cite{Good2013}, Goodarzi computed the Hilbert series of squarefree monomial ideals.
We compute the Hilbert series of symbolic powers of cover ideals of crown and complete multipartite graphs (Theorem 
\ref{Hil-crown}, Theorem \ref{C6.8}).

Computing and finding bounds for the multiplicity of homogeneous ideals have been studied by a 
number of researchers (see \cite{Bh1993}, \cite{Herzog'sBook}, \cite{villarreal_book}).
We compute the multiplicity of symbolic powers of cover ideals and edge ideals
in terms of combinatorial
invariants (Corollary \ref{multi-coverideal}).

In order to prove our main results, we first show that the minimal monomial generators of
symbolic powers of cover ideals
have specific order that satisfies some nice properties (Lemma \ref{techlemma}, Lemma \ref{C6.4}).
Using this ordering and certain exact sequences, we obtain  main results.

Our paper is organized as follows. 
In Section \ref{pre}, we 
collect the necessary notion, terminology and some results that are used in  rest of the article.
 The regularity and Hilbert series of symbolic powers of cover ideals 
 of crown and multipartite graphs are discussed in Sections \ref{crown} and \ref{complete}, respectively.
 The multiplicity of symbolic powers of edge ideals and cover ideals is 
 studied in Section \ref{multiplicity}.
 
 \section{Preliminaries}\label{pre}
In this section, we set up basic definitions, notation and some important results which are needed
for  rest of the paper. 

\subsection{Notion from commutative algebra} 
Let $M =\underset{k \in \mathbb{N}} \bigoplus M_k$ be a finite graded $S$-module.
The \emph{Hilbert series} of $M$, denoted by $H(M,t)$, is defined as 
$H({M},t):= \underset{k \in \mathbb{N}} \sum \dim_{\mathbb{K}}(M_k) t^k$. By \cite[Proposition 4.4.1]{Bh1993}, there exists a polynomial $h_M(t)\in \mathbb{Z}[t]$ such that $H(M,t)=\dfrac{h_M(t)}{(1-t)^d}$, where $d$ is the dimension of $M$. 
The \emph{multiplicity} of $M$, denoted by $e(M)$, is defined as $e(M)=h_M(1)$.
The 
\emph{Castelnuovo-Mumford regularity} of $M$, denoted by $\reg(M)$, is defined as 
$\reg(M)=\max \{j-i \mid \Tor_i^S(M,\mathbb{K})_j \neq 0\}$.

Let $I$ be an ideal in a Noetherian domain $R$. The $s$-th \emph{symbolic power} of $I$ is defined by
$I^{(s)}:= \bigcap\limits_{\mathfrak{p} \in \ass(R/I)}  (I^sR_{\mathfrak{p}} \cap R).$
It follows from \cite[Proposition 1.4.4]{Herzog'sBook} that if $I$ is a squarefree monomial ideal in $S$, then $s$-th 
symbolic power of $I$ is
$
I^{(s)}= \bigcap_{\mathfrak{p}\in \ass(S/I)} \mathfrak{p}^s.
$

\begin{remark}\label{symbCon} Let $\mathfrak{p}=(x_{i_1},\dots, x_{i_r})$. For a monomial $u$ in $S$, set $m_i(u)=\max\{j:x_i^j\mid u\}$ and $\deg_{\mathfrak{p}}(u)=\sum\limits_{k=1}^rm_{i_k}(u)$.  
	Let $I$ be a squarefree monomial ideal with $I=\bigcap\limits_{\mathfrak{p}\in \ass(S/I)} \mathfrak{p}$. Then  $u\in I^{(s)}$ if and only if $\deg_{\mathfrak{p}}(u)\geq s$ for all $\mathfrak{p} \in$ Ass$(S/I)$. 
\end{remark}

\subsection{Notion from combinatorics}
Let $G$ be a finite simple graph with  the vertex set $V(G)$
and edge set $E(G)$. 
A subset
$X$ of $V(G)$ is called \textit{independent} if for all $x,y\in X$,  $\{x,y\} \notin E(G)$. A graph $G$ is said to be \emph{bipartite} if there exist two disjoint independent sets $X$ and $Y$ such that $V(G)=X\cup Y$.
A graph $G$ is said to be \emph{complete multipartite} if $V(G)$ can be partitioned 
into sets $V_1,\ldots,V_k$ for some $k \geq 2$ such that
$E(G)=\bigcup_{i \neq j}\left\{ \{x,y\}\mid x\in V_i,y\in V_j\right\}$
and it is denoted by $K_{p_1,\dots, p_k}$, where $p_i=|V_i|$.
An $n$-\emph{crown graph} (or simply a \emph{crown graph}), denoted by $C_{n,n}$,  
is a bipartite graph on the vertex set $V(G)=\{x_1,\ldots,x_n,y_1,\ldots,y_n\}$ with  edge set 
$E(G)=\Big\{\{x_i,y_j\} \mid 1 \leq i,j\leq n, i \neq j \Big\}.$
A subset $C \subset V(G)$ is a \emph{vertex cover}  of
$G$ if for each $e \in E(G)$, $e\cap C \neq \emptyset$.  If $C$ is minimal
with respect to inclusion, then $C$ is called  a \textit{minimal vertex
cover} of $G$.
\begin{example} 
Let $G=K_{2,2,1,1}$ and $H=C_{4,4}$ be complete multipartite graph 
and crown graph on $\{x_{1,1},x_{1,2},x_{2,1},x_{2,2},x_{3,1}
,x_{4,1}\}$ and $\{x_1,\ldots,x_4,y_1,\ldots,y_4\}$ as given in the figure below.
 
 \begin{minipage}{\linewidth}
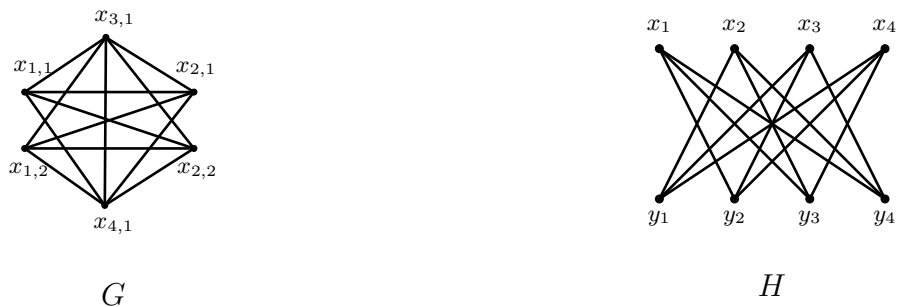
\begin{minipage}{0.45\linewidth}
\begin{figure}[H]
\definecolor{ududff}{rgb}{0.30196078431372547,0.30196078431372547,1}
\begin{tikzpicture}[scale=0.75]
\draw [line width=1pt](1,2)-- (2.42,0.99);
\draw [line width=1pt] (1,3)-- (4,3);
\draw [line width=1pt] (4,3)-- (2.42,0.99);
\draw [line width=1pt] (4,2)-- (2.44,3.97);
\draw [line width=1pt] (2.44,3.97)-- (1,2);
\draw [line width=1pt] (4,2)-- (1,3);
\draw [line width=1pt] (1,2)-- (4,2);
\draw [line width=1pt] (1,2)-- (4,3);
\draw [line width=1pt] (1,3)-- (2.42,0.99);
\draw [line width=1pt] (1,3)-- (2.44,3.97);
\draw [line width=1pt] (4,2)-- (2.42,0.99);
\draw [line width=1pt] (4,3)-- (2.44,3.97);
\draw [line width=1pt] (2.44,3.97)-- (2.42,0.99);
\begin{scriptsize}
\draw [fill=black] (1,2) circle (1.5pt);
\draw[color=black] (1.08,1.6) node {$x_{1,2}$};
\draw [fill=black] (2.42,0.99) circle (1.5pt);
\draw[color=black] (2.58,.62) node {$x_{4,1}$};
\draw [fill=black] (1,3) circle (1.5pt);
\draw[color=black] (1.16,3.42) node {$x_{1,1}$};
\draw [fill=black] (4,3) circle (1.5pt);
\draw[color=black] (4.04,3.4) node {$x_{2,1}$};
\draw [fill=black] (4,2) circle (1.5pt);
\draw[color=black] (4.06,1.6) node {$x_{2,2}$};
\draw [fill=black] (2.44,3.97) circle (1.5pt);
\draw[color=black] (2.6,4.3) node {$x_{3,1}$};
\end{scriptsize}
\end{tikzpicture}	
\caption*{$G$}
\end{figure}
\end{minipage}
\begin{minipage}{0.6\linewidth}

\begin{figure}[H]
\definecolor{ududff}{rgb}{0.30196078431372547,0.30196078431372547,1}
\begin{tikzpicture}[scale=1.]
\draw [line width=1pt] (1,4)-- (2,2);
\draw [line width=1pt] (2,2)-- (3,4);
\draw [line width=1pt] (3,4)-- (1,2);
\draw [line width=1pt] (1,2)-- (2,4);
\draw [line width=1pt] (2,4)-- (4,2);
\draw [line width=1pt] (4,2)-- (3,4);
\draw [line width=1pt] (3,2)-- (4,4);
\draw [line width=1pt] (4,4)-- (2,2);
\draw [line width=1pt] (3,2)-- (2,4);
\draw [line width=1pt] (1,4)-- (3,2);
\draw [line width=1pt] (1,2)-- (4,4);
\draw [line width=1pt] (4,2)-- (1,4);
\begin{scriptsize}
\draw [fill=black] (1,4) circle (1.5pt);
\draw[color=black] (1,4.3) node {$x_1$};
\draw [fill=black] (2,2) circle (1.5pt);
\draw[color=black] (2,1.75) node {$y_2$};
\draw [fill=black] (3,4) circle (1.5pt);
\draw[color=black] (3,4.3) node {$x_3$};
\draw [fill=black] (1,2) circle (1.5pt);
\draw[color=black] (1,1.75) node {$y_1$};
\draw [fill=black] (2,4) circle (1.5pt);
\draw[color=black] (2,4.3) node {$x_2$};
\draw [fill=black] (4,2) circle (1.5pt);
\draw[color=black] (4,1.75) node {$y_4$};
\draw [fill=black] (3,2) circle (1.5pt);
\draw[color=black] (3,1.75) node {$y_3$};
\draw [fill=black] (4,4) circle (1.5pt);
\draw[color=black] (4,4.3) node {$x_4$};
\end{scriptsize}
\end{tikzpicture}
\caption*{$H$}
\end{figure}
\end{minipage}
\end{minipage}

It can be noted that $\{x_{1,1},x_{1,2},x_{2,1},x_{2,2},x_{3,1}\}$ 
and 
$\{x_2,x_3, x_4,y_2,y_3,y_4\}$ are minimal vertex covers of $K_{2,2,1,1}$ and $C_{4,4}$, respectively.
 
\end{example}
For any undefined terminology and further basic definitions, we refer the reader to \cite{Bh1993}, \cite{Herzog'sBook}.

\section{Crown graph}\label{crown}
In this section, we study the regularity and Hilbert series of symbolic powers of cover ideals of
crown graphs. Throughout this section, $G$ denotes a  crown graph.
\subsection{Regularity}
In this subsection, we obtain the linear function for the regularity of $J(G)^{(s)}$ for all 
$s \geq 1$. 
Our result Theorem \ref{basecase} shows that $\reg(J(G)^{(s)})$ is a 
linear function with the stabilization index $s_0=1$ and free constant $b=0$.
In order to prove this, we first fix certain notation.
\begin{notation} \label{setup}
	For $n\geq 3$, let $G=C_{n,n}$ be a graph with $V(G)=\{x_1,\ldots,x_n,y_1,\ldots,y_n\}$. Set 
	\begin{align*}
	M_x=\prod_{i=1}^n x_i, ~~  M_y=\prod_{i=1}^n y_i,~~
	M=M_xM_y \text{ and }M_i=\frac{M}{x_iy_i} \text{ for $1 \leq i \leq n$. }
	\end{align*}
\end{notation}
First, we find the monomial generating set of
cover ideal of crown graph.

\begin{lemma}
	Let $G=C_{n,n}$ with notation as in \ref{setup}. Then 
	$J(G)=(M_x,M_y,M_1,\ldots,M_n).$ 
	In particular, $\deg(J(G))=2n-2$.
\end{lemma}
\begin{proof}
	Since $J(G)=\bigcap\limits_{i\neq j}(x_i,y_j)$, $(M_x,M_y,M_1,\ldots,M_n) \subset J(G)$. 
	Let $u$ be a monomial in $J(G)$. If either $M_x \mid u$ or $M_y\mid u$, then we are done. Now, we assume that $M_x\nmid u$ and $M_y\nmid u$. This forces that there exist $i$ and $j$ such that $x_i\nmid u$ and $y_j\nmid u$. For $k\neq i$, $J(G)\subset (x_i,y_k)$ which forces that $y_k\mid u$. This implies that $j=i$, and by symmetry for $k\neq i$, $ x_k\mid u$. Therefore, $M_i\mid u$ which gives the desired result.
\end{proof}
For a monomial $u$, \emph{support} of $u$, denoted by $\supp(u)$, is defined as $\supp(u)=\{x_i:x_i\mid u\} $.
The following lemma summarizes some basic properties of $J(G)^{(s)}$.
\begin{lemma}\label{techlemma} Let $G=C_{n,n}$ with notation as in \ref{setup}. Then, for $s \geq 2$,
	\begin{enumerate}[\rm i)]
		\item $(M_x):M_y=(M_x)$.
		\item $(M_x):M_i=(x_i)$ and $M_y:M_i=(y_i)$ for $1 \leq i \leq n$.
		\item $(M_j):M_i=(x_iy_i)$ for $i\neq j$.
		\item $(M_x,M_y,M_1,\ldots,M_{i-1}):M_{i}=(x_i,y_i)$ for $1 \leq i \leq n$.
		\item $J(G)^{s}:M_x=J(G)^{s-1}$.
		\item  $(J(G)^{s},M_x):M_y=\left( J(G)^{s-1}, M_x\right) $.
		\item $\left(J(G)^{s},M_x,M_y,M_1,\ldots,M_{i-1}\right):M_i=\left(x_i,y_i,M_i^{s-1}\right)$ for $1\leq i\leq n$.
	\end{enumerate}
\end{lemma}
\begin{proof} (i)-(iv) are standard.\\
	(v) The assertion follows from  \cite[Lemma 3.2]{Seyed16}.\\
	(vi) Since $\supp(M_x) \cap \supp(M_y)=\emptyset$, by (v), $(J(G)^{s},M_x):M_y=\left( J(G)^{s-1}, M_x\right) $.\\
	(vii) By (iv), $(J(G)^{s},M_x,M_y,M_1,\ldots,M_{i-1}):M_i \supset (x_i,y_i,M_i^{s-1})$. Let 
	$u$ be a monomial in $(J(G)^{s},M_x,M_y,M_1,\ldots,M_{i-1}):M_i$. If either $x_i \mid uM_{i}$ or
	$y_i \mid uM_i$, then $u \in (x_i,y_i, M_i^{s-1})$. Suppose $x_i \nmid uM_i$ and $y_i \nmid uM_i$.
	Note that $(M_x,M_y,M_1,\dots,M_{i-1})\subset (x_i,y_i)$ and $x_i \nmid uM_i$ and $y_i \nmid uM_i$ forces that $uM_i\in J(G)^s $. Since $G$ is a bipartite graph, by \cite[Corollary 2.6]{GRV05}, $J(G)^{(s)}=J(G)^s$, and hence $J(G)^{s}=\bigcap\limits_{i\neq j}(x_i,y_j)^s$. For $k\neq i$, $J(G)^s\subset (x_i,y_k)^s$, and $x_i \nmid uM_i$ which implies that $y_k^s\mid uM_i$. Note that for $k\ne i$, $y_k^2\nmid M_i$ which forces that $y_k^{s-1}\mid u$. Similarly, we get that $x_k^{s-1}\mid u$ for all $k\neq i$. Hence $M_i^{s-1} \mid u$.
\end{proof}
We now proceed to compute the regularity of powers of $J(G)$.
\begin{theorem}\label{basecase}  Let $G=C_{n,n}$. Then for all $s \geq 1$,
	\[
	\reg(J(G)^s)=s \cdot \deg(J(G)).
	\] 
\end{theorem}
\begin{proof}
	It follows from \cite[Lemma 3.1]{Seyed} that $s\cdot\deg(J(G))\leq \reg(J(G)^s)$. We need to prove that
	$\reg(J(G)^s) \leq s\cdot \deg(J(G))$. By \cite[Proposition 8.1.10]{Herzog'sBook}, $\reg(J(G))=\pd(S/I(G))$.  
	If $s=1$, then the result follows from \cite[Theorem 4.3]{Rather2018}. So, assume that $s \geq 2.$
	Consider the following short exact sequence:
	\begin{equation}\label{exactCrown}
	0 \longrightarrow \dfrac{S} {J(G)^s:M_x}(-n) \longrightarrow \dfrac{S}{J(G)^s} \longrightarrow 
	\dfrac{S}{(J(G)^s,M_x)} \longrightarrow 0.
	\end{equation}
	By Lemma \ref{techlemma}, $J(G)^s:M_x=J(G)^{s-1}$.
	Then, by induction, $$\reg \left( J(G)^{s-1}(-n)\right)\leq (s-1)\cdot\deg(J(G))+n \leq s\cdot\deg(J(G)).$$
	Now, by Equation \eqref{exactCrown}, it is sufficient to show that $\reg(J(G)^s,M_x)) \leq s\cdot\deg(J(G)).$
	\vskip 2mm \noindent
	\textbf{Claim:} $\reg(J(G)^s,M_x) \leq s \cdot \deg(J(G))$ for all $s \geq 1$.
	\vskip 1mm \noindent
	\textit{Proof of the claim:}
	We proceed by induction on $s$. If $s=1$, then $(J(G),M_x)=J(G)$ and the result follows from \cite[Theorem 4.3]{Rather2018}. Assume that $s \geq 2$. 
	Set $K=(J(G)^s,M_x)$, $K_1=(K,M_y)$ and for $2\leq l\leq n+1$, $K_l=(K_{l-1},M_{l-1})$. Note that $K_{n+1}=J(G)$.
	Consider the following short exact sequences: 
	\begin{eqnarray}\label{exact1}
	0 \longrightarrow \frac{S}{ K:M_{y}}(-n) \longrightarrow \frac{S}{K} \longrightarrow 
	\frac{S}{K_{1}} \longrightarrow 0,
	\end{eqnarray}
	for $1 \leq l \leq n,$
	\begin{eqnarray}\label{exact2}
	0 \longrightarrow \frac{S}{ K_l:M_{l}}(-(2n-2)) \longrightarrow \frac{S}{K_l} \longrightarrow 
	\frac{S}{K_{l+1}} \longrightarrow 0 . 
	\end{eqnarray}
	Using Equations \eqref{exact1} and \eqref{exact2}, we get 
	\begin{eqnarray*}
		\reg(K) \leq  \max\left\{ \reg(K:M_y)+n,~\reg(J(G)),~\reg(K_l:M_{l})+2n-2 \text{ for }1\leq l \leq n\right\}.
	\end{eqnarray*}
	We now prove that each of regularities appearing on the right hand side of the above inequality is 
	bounded above by $s \cdot \deg(J(G))$. By Lemma \ref{techlemma}, Theorem \ref{basecase} and 
	\cite[Theorem 4.3]{Rather2018}, 
	we have
	\begin{align*}
	\reg(K:M_y)=&\reg(J(G)^{s-1},M_x),~\reg(J(G))= \deg(J(G))\\
	\reg(K_l:M_{l})=&\reg(x_{l},y_{l}, M_{l}^{s-1}), \text{ for all $1\leq l \leq n$}.
	\end{align*}
	By induction, $\reg(K:M_y)\leq(s-1)\cdot\deg(J(G))$. Since $x_l,y_l,M_l^{s-1}$ is a regular sequence 
	with $\deg(x_l)=\deg(y_l)=1$ and $\deg(M_l)^{s-1}=(s-1)\cdot\deg(J(G))$, 
	$\reg(K_l:M_l)=(s-1)\cdot \deg(J(G))$. 
	Therefore, $\reg(J(G)^s,M_x) \leq s\cdot\deg(J(G))$. 
\end{proof}

\subsection{Hilbert series.}
We compute the Hilbert series of symbolic powers of cover ideals of crown graphs. We begin by computing the Hilbert series of cover ideal.
\begin{theorem}\label{hilbertscrown}
	Let $G=C_{n,n}$ for $n \geq 3$ with notation as in \ref{setup}. Then 
	\[ H\left(\dfrac{S}{J(G)},t\right)= \frac{\displaystyle \sum_{i=0}^{n-1} (i+1)t^i + \displaystyle \sum_{i=0}^{n-3}(n-i-1)t^{n+i} - (n-1)t^{2n-2}}{(1-t)^{2n-2}}.
	\]
\end{theorem}
\begin{proof}
	Set $I_0=(M_x,M_y)$ and $I_{i}=(I_{i-1},M_i)$ for all $1 \leq i\leq n$.	
	For $1\leq i\leq n$, consider the exact sequence:
	\begin{equation*}
	0 \longrightarrow \frac{S}{I_{i-1}:M_i}(-(2n-2))
	{\longrightarrow} \frac{S}{I_{i-1}} \longrightarrow
	\frac{S}{I_i} \longrightarrow 0.\nonumber
	\end{equation*}
	We have
	$H\left(\frac{S}{J(G)},t \right) 
	= H\left(\frac{S}{I_0},t\right)-t^{2n-2}\sum_{i=1}^n H\left(\frac{S}{I_{i-1}:M_i},t\right).$
	Since $M_x,M_y$ is a regular sequence on $S$ of degree $n$, 
	we get $H\left(\dfrac{S}{I_0},t\right)=\dfrac{(1-t^n)^2}{(1-t)^{2n}}.$
	By Lemma \ref{techlemma}, $I_{i-1}:M_i=(x_i,y_i)$ for any $1 \leq i \leq n$
	which implies that 
		$H\bigg( \frac{S} {I_{i-1}:M_i},t \bigg)=\frac{1}{(1-t)^{2n-2}}.$
	Hence 
	\begin{eqnarray*}
		H\left(\dfrac{S}{J(G)},t\right) = 
		\frac{ \displaystyle \sum_{i=0}^{n-1} (i+1)t^i + \displaystyle \sum_{i=0}^{n-3} (n-i-1)t^{n+i}-(n-1)t^{2n-2}}{(1-t)^{2n-2}}.
	\end{eqnarray*}
	
\end{proof}
We end this section by proving the following main result.
\begin{theorem}\label{Hil-crown}
	Let $G=C_{n,n}$ for $n \geq 3$ with notation as in \ref{setup}. Then for all $s \geq 1$, $H\left( \frac{S}{J(G)^s},t \right) $
	\[
	= \frac{ \displaystyle \sum_{i=0}^{ns-1} (i+1)t^i + \displaystyle \sum_{i=0}^{n-3}(n-i-1)st^{ns+i} - (n-1)st^{ns+n-2} - \displaystyle \sum_{i=0}^{s-2} (i+1)nt^{s(2n-2)-i(n-2)}}{(1-t)^{2n-2}}.
	\]
\end{theorem}
\begin{proof}
	We proceed by induction on $s$. By Theorem \ref{hilbertscrown}, the result is true for $s=1$.
	Assume that $s\geq 2$.
	Using Lemma \ref{techlemma} (v) and exact sequence \eqref{exactCrown}, we get 
	\begin{align}\label{main-eq}
	H\left( \frac{S}{J(G)^s},t \right) = t^n H\left( \frac{S}{J(G)^{s-1}},t \right) + H\left( \frac{S} {\left( J(G)^s,M_x \right)},t \right).
	\end{align}
	\vskip 2mm \noindent
	\textbf{Claim:} For all $s \geq 1$, $H\bigg( \frac{S} {\big( J(G)^s,M_x \big)},t \bigg) $
	$$ =\frac{\displaystyle \sum_{i=0}^{n-1} 
		(i+1)t^i+n \displaystyle \sum_{i=n}^{ns-1}t^i + \displaystyle \sum_{i=1}^{n-2}(n-i)t^{ns+i-1} - 
		(n-1)t^{ns+n-2} - n \displaystyle \sum_{i=0}^{s-2}t^{s(2n-2)-i(n-2)}}{(1-t)^{2n-2}}.$$
	Now, it is enough to prove the above claim as the desired result follows from induction argument, claim and Equation \eqref{main-eq}.
	\\ \textit{Proof of the claim:}
	For $s=1$ the result follows from Theorem \ref{hilbertscrown} and 
	the fact that $(J(G),M_x)=J(G)$.
	Assume that $s \geq 2$. 
	Using Equations \eqref{exact1} and \eqref{exact2}, we get\
	\begin{equation}\label{hilbertSeq}
	H\left(\dfrac{S}{K},t\right)= t^nH\left(\dfrac{S}{K:M_y },t\right)+\sum_{i=1}^{m}t^{2n-2}H\left(\dfrac{S}{K_i:M_i},t\right)+H\left(\dfrac{S}{J(G)},t\right).
	\end{equation} 
	
	By Lemma \ref{techlemma}(vi) and (vii),  $K:M_y=(J(G)^{s-1}, M_x)$ and $K_i:M_i=(x_i,y_i, M_i^{s-1})$.
	Since $x_i, y_i, M_i^{s-1}$ is a regular sequence with $\deg(M_i)=2n-2$, we get that 
	$$H\left(\dfrac{S}{K_i:M_i},t\right)=\dfrac{(1-t^{(s-1)(2n-2)})}{(1-t)^{2n-2}}.$$ 
       Now the claim follows from Equation \eqref{hilbertSeq}, 
	Theorem \ref{hilbertscrown} and induction.	
\end{proof}

\section{Complete multipartite graph}\label{complete}
In this section, we study the regularity and Hilbert series of symbolic powers of cover
ideals of complete multipartite graphs. Throughout this section, $G$ denotes a  complete multipartite graph.
\subsection{Regularity}
We determine the regularity of symbolic powers of cover ideals of 
 complete multipartite graphs. In order to compute the regularity of 
$J(G)^{(s)}$, we first find the generators of $J(G)$ and its symbolic powers.
We begin by fixing some notation which are used for the rest of this section.

\begin{notation}\label{multi-nota}
	Let $G=K_{p_1,\dots,p_k}$ be a complete multipartite graph with the vertex set
	$$V(G)=\bigcup_{i=1}^{k} \{x_{ij}: 1\leq j\leq p_i \},~~ p_1\geq p_2\geq\dots\geq p_k\geq 1 \text{ and }
	k\geq 2.$$ 
	Set
	$$n=p_1+p_2+\dots+p_k,~ M_i=\prod_{j=1}^{p_i}x_{ij} \text{ for }
	1\leq i\leq k,~M=\prod_{i=1}^{k} M_i \text{ and }
	N_i=\frac{M}{M_i} \text{ for } 1\leq i\leq k.$$ 
\end{notation}

The following lemma describes the minimal monomial generating set  of $J(G)$.
\begin{lemma} Let $G=K_{p_1,\dots,p_k}$ with the notation as in \ref{multi-nota}. Then $J(G)=(N_i \mid 1\leq i\leq k)$. In particular, $\deg(J(G))=n-p_k$.
\end{lemma}
\begin{proof}
	Note that $\supp(N_i)$ is a vertex cover of $G$ which implies that $(N_i\mid 1\leq i\leq k)\subset J(G)$. Let $u$ be a monomial in $J(G)$. 
	If for all $i,j$, $x_{i,j}\mid u$, then $u\in (N_i \mid 1\leq i\leq k)$. Now, without loss of generality, assume that $x_{1,1}\nmid u$. 
	Since for all 
	$i\neq 1$ and $j$, $J(G)\subset (x_{1,1},x_{i,j}) $, we get $x_{i,j}\mid u$, which further implies that $N_1\mid u$. This completes the proof.
\end{proof}
If $k=2$, then by \cite[Corollary 2.6]{GRV05},
$J(G)^s=J(G)^{(s)}$ for all $s \geq 1$. If $k \geq 3$, then 
$G$ is non-bipartite graph and every vertex in $G$ is adjacent to every odd cycle in $G$. 
Therefore, by \cite[Theorem 4.9 and Remark 4.10]{DrabGue},
$J(G)^{(s)}=MJ(G)^{(s-2)}+J(G)^s$.
Now, we further reduce the above expression.
\begin{lemma}\label{symbolicGen}
	Let $G=K_{p_1,\dots,p_k}$ with the notation as in \ref{multi-nota}.
	Then $$J(G)^{(s)}=MJ(G)^{(s-2)}+\left( N_j^s\mid j\in[k]\right).$$
\end{lemma}
\begin{proof} For $e\in E(G)$, $\mathfrak{p}_e$ denote an ideal generated by end points of $e$. Let $u$ be a monomial $J(G)^{(s)}$. By Remark \ref{symbCon}, for every $e\in E(G)$, we have $\deg_{\mathfrak{p}_e}(u)\geq s$. Note that $\deg_{\mathfrak{p}_e}(M)=2$. If $u=Mv$, then $\deg_{\mathfrak{p}_e}(v)\geq s-2$ which implies that $v\in J(G)^{(s-2)}$.
	
	Suppose $M \nmid u$. Then there exists $x_{i,j}$ such that $x_{i,j}$ does not 
	divide $u$. Without loss of generality, we may assume that $x_{1,1} \nmid u$.
	Since $x_{1,1}$ is adjacent to $x_{i,j}$ for $i\geq 2$ and for all  $j$, by Remark \ref{symbCon}, $N_1^s$ divides $u$. Thus, we have $J(G)^{(s)}\subset MJ(G)^{(s-2)}+
	\left( N_j^s: j\in[k]\right)$. Clearly, $\left( N_j^s: j\in[k]\right)\subset J(G)^{(s)}$. It follows from Remark \ref{symbCon} that $MJ(G)^{(s-2)} \subset J(G)^{(s)}$
	which completes the proof.
\end{proof}

For a monomial ideal $I=(m_1,\ldots,m_r)$, let $I^{[s]}$ denote an ideal generated by
$m_1^s,\ldots,m_r^s$. 
The following lemma plays a crucial role to compute the regularity of $J(G)^{(s)}$.
\begin{lemma}\label{C6.4} Let $G=K_{p_1,\dots, p_k}$ with the notation as in \ref{multi-nota}. Then
	\begin{enumerate}[\rm i)]
		\item $J(G)^{(s)}:M=J(G)^{(s-2)}$ for $s\geq 2$.
		\item $(J(G)^{(s)},M)=(J(G)^{[s]},M)$ for $s\geq 1$.
		\item $(N_1^s,\dots,N_{i-1}^s):N_i^s=(M_i^s)$ for $2\leq i\leq k$.
		\item $J(G)^{[s]}:M=J(G)^{[s-1]}$ for $s\geq 2$.
	\end{enumerate}
\end{lemma}
\begin{proof}
	(i) and (ii) follow from \cite[Lemma 3.4]{Fakhari17} and Lemma \ref{symbolicGen}, respectively.\\
	(iii) Clearly $(M_i^s)\subset (N_1^s,\dots,N_{i-1}^s):N_i^s$.
	Let $u$ be a monomial in $(N_1^s,\dots,N_{i-1}^s):N_i^s$.
	This forces that for some $1\leq j<i$, $N_j^s\mid uN_i^s$ which implies that $M_i^s \mid u$. \\
	(iv) If $u \in J(G)^{[s]}:M$, then there exists $i$ such that $N_i^s\mid  uM$ and so $N_i^{s-1} \mid u$.
	On the other side, since  $N_i\mid M$ for all $i$, we get $N_i^s\mid MN_i^{s-1}$.
\end{proof}
For fixed $s\ge 2$ and $j\in [k]$, we associate  an ideal $I_{s,j}=\left(M, N_1^s,\ldots,N_j^s \right)$.
Now, we compute the regularity of $I_{s,j}$ in terms of $s$ and $p_j$, which helps to compute the regularity of $J(G)^{(s)}$.

\begin{lemma}\label{regLem} For fixed $s\geq 2$ and $j\in [k]$, $\reg\left(I_{s,j}\right)=s(n-p_j)+p_j-1.$
\end{lemma}
\begin{proof}
	We prove the assertion by induction on  $j$. Suppose $j=1$, consider the exact sequence
	\[
	0\longrightarrow \dfrac{S}{M:N_1^s}(-s(n-p_1)) \longrightarrow
	\dfrac{S}{ M }\longrightarrow \dfrac{S}{\left( M,N_1^s\right)}\longrightarrow 0. 
	\]
	Note that $\reg \left(  M \right)=n$  and 
	$\left( M:N_1^s\right)= (M_1)$. 
	Therefore,
	$$\reg\left(\left( M:N_1^s\right) (-s(n-p_1))\right)=p_1+s(n-p_1).$$ 
	Since $s \geq 2$, it follows from \cite[Lemma 1.2 (v)]{HTT} that
	$\reg\left( M,N_1^s \right)=s(n-p_1)+p_1-1.$
	Consider the exact sequence
	\[
	0\longrightarrow \dfrac{S}{I_{s,j-1}:N_j^s}(-s(n-p_j)) 
	\longrightarrow
	\dfrac{S}{I_{s,j-1} }\longrightarrow \dfrac{S}{I_{s,j}}\longrightarrow 0. 
	\]
	Note that $(I_{s,j-1}:N_j^s)= (M_j)$, and 
	hence $\reg\left(\left(I_{s,j-1}:N_j^s\right)(-s(n-p_j))\right)=s(n-p_j)+p_j$.
	By induction, $\reg\left(I_{s,j-1}\right)=s(n-p_{j-1})+p_{j-1}-1.$
	Hence by \cite[Lemma 1.2 (v)]{HTT}, we get 
	$
	\reg\left(I_{s,j}\right)=s(n-p_j)+p_j-1.
	$
\end{proof}
We now compute the regularity of $J(G)^{(s)}.$ Since complete multipartite graph is a matroid, there is another way to compute the regularity of $J(G)^{(s)}$, see \cite[Theorem 4.5]{MT17}. We have provided here an elementary proof so
that the result is accessible to readers who are not familiar with matroid.
\begin{theorem}\label{powerMultiThm}
	Let $G=K_{p_1,\dots,p_k}$ with the notation as in \ref{multi-nota}.
	Then for all $s \geq 1$, $$\reg\left(J(G)^{(s)} \right)=s\cdot \deg(J(G))+p_k-1.$$
\end{theorem}
\begin{proof}
	We prove the result by induction on $s$. If  $s=1$, then the result follows from \cite[Theorem 5.3.8]{PhDT103J}. Assume that $s>1$. Consider the following exact sequence:
	\begin{equation}\label{cmpMltSES}
	0\longrightarrow \dfrac{S}{J(G)^{(s)}:M}(-n)\longrightarrow \dfrac{S}{J(G)^{(s)}}\longrightarrow 
	\dfrac{S}{(J(G)^{(s)}, M)}\longrightarrow0.
	\end{equation}
	By Lemma \ref{C6.4}(i), $J(G)^{(s)}:M=J(G)^{(s-2)}$ and by induction $$\reg\left(\left(J(G)^{(s)}:M\right)(-n)\right)=(s-2)\cdot \deg(J(G))+n+p_k-1. $$
	It follows from Lemma \ref{symbolicGen} that $(J(G)^{(s)},M)=I_{s,k}$. Now, by  Lemma \ref{regLem}, 
	we get $$\reg\left(J(G)^{(s)}, M\right)=s\cdot \deg(J(G))+p_k-1.$$ 
	Hence the assertion follows \cite[Lemma 1.2]{HTT}. 
\end{proof}
It follow from  Theorem \ref{powerMultiThm} that if $p_k=1$, then the free constant $b=0$. 
\begin{corollary}
	Let $G$ be a complete graph on $n$ vertices. Then, for all $s \geq 1,$
	$$\reg\left(J(G)^{(s)}\right)=s\cdot \deg(J(G))=s(n-1).$$
\end{corollary}
\subsection{Hilbert series}
In this subsection, we compute the Hilbert series of
symbolic powers of $J(G)$ in terms of 
number of vertices and cardinality of partition.
To accomplish this, we  first study the Hilbert series of 
$\frac{S}{J(G)^{[s]}}$ for all $s \geq 1$.

\begin{proposition}\label{C6.5}
	Let $G=K_{p_1,\dots, p_k}$ with the notation as in \ref{multi-nota}. Then, for all $s \geq 1,$
	\[
	H\left(\frac{S}{J(G)^{[s]}},t\right) = \frac{1- \displaystyle\sum_{i=1}^{k} t^{s(n-p_i)}+(k-1)t^{sn}}{(1-t)^n}.
	\]
\end{proposition}
\begin{proof}
	By Lemma \ref{C6.4}(iii), $(N_1^s,\dots,N_{i-1}^s):N_i^s=(M_i^s)$ for all  $2\leq i\leq k$. Now, for $2\leq i\leq k$ consider the exact sequences: 
	\begin{eqnarray*}
		0\longrightarrow \frac{S}{(M_i^s)}(-s(n-p_i)) \longrightarrow \frac{S}{(N_1^s,\dots,N_{i-1}^s)} \longrightarrow \frac{S}{(N_1^s,\dots,N_i^s)} \longrightarrow 0.
	\end{eqnarray*}
	We know that
	$H\left(\frac{S}{(M_i^s)},t\right)=\frac{1-t^{sp_i}}{(1-t)^n} \text{ for all
		$2\leq i\leq k$.}$
	Therefore, by applying successively the above short exact sequences, we get
	\begin{eqnarray*}
		H\left(\frac{S}{J(G)^{[s]}},t\right) &=&  H\left(\frac{S}{(N_1^s)},t\right) -\sum_{i=2}^kt^{s(n-p_i)}H\left(\dfrac{S}{(M_i^s)},t\right) 
		\\  &=& \dfrac{1-t^{s(n-p_1)}}{(1-t)^n}
		-\sum\limits_{i=2}^k\left( \frac{t^{s(n-p_i)}-t^{sn}}{(1-t)^n} \right) 
		\\ 
		&=& \frac{1-\displaystyle \sum_{i=1}^k t^{s(n-p_i)}+(k-1)t^{sn}}{(1-t)^n}.
	\end{eqnarray*}
\end{proof}

To obtain the Hilbert series of $\frac{S}{J(G)^{(s)}}$, we need the following lemma:
\begin{lemma}\label{C6.7}
	Let $G=K_{p_1,\dots, p_k}$ with the notation as in \ref{multi-nota}. Then for all $s \geq 1,$
	\[
	H\left(\frac{S}{(J(G)^{[s]},M)},t\right) = \frac{1-t^n- \displaystyle\sum_{i=1}^{k} t^{s(n-p_i)}+\displaystyle\sum_{i=1}^{k} t^{s(n-p_i)+p_i}}{(1-t)^n}.
	\]
\end{lemma}
\begin{proof}
	Consider the  short exact sequence: 
	\begin{equation}\label{comMulSES1}
	0\longrightarrow \frac{S}{J(G)^{[s]}:M}(-n) \longrightarrow \frac{S}{J(G)^{[s]}} \longrightarrow \frac{S}{(J(G)^{[s]},M)}\longrightarrow 0.
	\end{equation}
	By Lemma \ref{C6.4}(iv), $J(G)^{[s]}:M=J(G)^{[s-1]}$. Therefore 
	\begin{eqnarray*}
		H\left(\frac{S}{(J(G)^{[s]},M)},t\right)= H\left(\frac{S}{J(G)^{[s]}},t\right)- t^nH\left(\frac{S}{J(G)^{[s-1]}},t\right).
	\end{eqnarray*}
	Using Proposition \ref{C6.5}, we get the result.
\end{proof}
We are now ready to establish the Hilbert series of $\dfrac{S}{J(G)^{(s)}}$ for all $s \geq 1$.
\begin{theorem}\label{C6.8}
	Let $G=K_{p_1,\dots, p_k}$ with the notation as in \ref{multi-nota}. 
	Then $H\left(\dfrac{S}{J(G)^{(s)}},t\right)$ is 	
	\[
	\left\{\begin{array}{cc}\dfrac{1-t^{rn}+ \sum\limits_{j=0}^{r-1}\sum\limits_{i=1}^{k} (t^{p_i}-1)\ t^{(s-j)(n-p_i)+jp_i}}{(1-t)^n},
	& \text{ if } s=2r, r\geq 1\\
	\dfrac{1+(k-1)t^{(r+1)n}-  \sum\limits_{i=1}^{k} t^{(n-p_i)+rn} + \sum\limits_{j=0}^{r-1}\sum\limits_{i=1}^{k} (t^{p_i}-1)\ t^{(s-j)(n-p_i)+jp_i}}{(1-t)^n} &\text{ if } s=2r+1, r\geq 0.
	\end{array}\right.
	\]
\end{theorem}
\begin{proof}
	It follows from Lemma \ref{C6.4}(i) and (ii) that for $s\geq 2$, $J(G)^{(s)}:M=J(G)^{(s-2)}$ and $(J(G)^{(s)},M)=(J(G)^{[s]},M).$ Using Equation \eqref{cmpMltSES}, we get
	\begin{eqnarray}\label{addHil}
	H\left(\frac{S}{J(G)^{(s)}},t\right) = t^n H\left(\frac{S}{J(G)^{(s-2)}},t\right) + H\left(\frac{S}{(J(G)^{[s]},M)},t\right).
	\end{eqnarray}
	
	Suppose $s=2r$.
	We prove this by induction on $r$. 
	If $r=1$, then, by Lemma \ref{symbolicGen}, $J(G)^{(2)}=(J(G)^{[2]},M)$. Now the result follows from 
	Lemma \ref{C6.7}. 
	Assume that $r \geq 2$.
	Now by induction and Lemma \ref{C6.7}, we get the assertion.
		
	Suppose $s=2r+1$. By Lemma \ref{C6.5}, the result holds for $r=0$. Now assume that $r \geq 1$. The assertion follows from induction, Lemma \ref{C6.7} and Equation \eqref{addHil}.
\end{proof}

\section{Multiplicity}\label{multiplicity}
In this section, we study the multiplicity of symbolic powers of cover ideals and edge ideals.
The following lemma is probably well-known. We include it for the sake of completeness.
\begin{lemma}\label{multLem}
	Let $I$ be minimally generated by $h$ linear forms. Then  $e\left(\dfrac{S}{I^s}\right)=\displaystyle \binom{s+h-1}{h}.$ 
\end{lemma}
\begin{proof}
	Let $x\in \mathfrak{m}\setminus I$ be a generator of $\mathfrak{m}$,
	where $\mathfrak{m}=(x_1,\ldots,x_n)$ is the unique homogeneous maximal ideal in $S$.
	 Then $x$ is a regular element on $I$, and so on $I^s$. This gives  $H\left(\dfrac{S}{(I^s,x)},t\right)=(1-t)H\left(\dfrac{S}{I^s},t\right)$, and hence $e\left(\dfrac{S}{(I^s,x)}\right)=e\left(\dfrac{S}{I^s}\right).$ Therefore, without loss of generality, we may assume that $I=\mathfrak{m}$. Thus, we get $e\left(\dfrac{S}{I^s}\right)=\dim_{\mathbb{K}}\left(\dfrac{S}{I^s}\right).$
\end{proof}
\begin{obs}\label{rem-multi}
Let $I$ be a squarefree monomial ideal in $S$.
 Since the minimal associated primes of squarefree monomial ideals
 are generated by subsets of variables, by \cite[Corollary 4.7.8]{Bh1993} and 
Lemma \ref{multLem}, we have $$e\left(\dfrac{S}{I^{(s)}}\right)=\binom{h+s-1}{h}\left|\minh(I)\right|,$$
	where $\minh(I)=\{\mathfrak{p} \in \ass(S/I) : \hgt(\mathfrak{p})=\hgt(I)\}$ and
	$h=\hgt(I)$
\end{obs}
As a consequence of Observation \ref{rem-multi}, we obtain the multiplicity of symbolic powers of 
edge ideals and cover ideals in terms of combinatorial invariants.
\begin{corollary}\label{multi-coverideal}
Let $G$ be a graph and $h$ be the size of the smallest vertex cover of $G$. Then for all $s \geq 1$,
\begin{enumerate}[\rm i)]
 \item $e\left(\dfrac{S}{I(G)^{(s)}}\right)=\binom{h+s-1}{h} \mathcal{V}(G)$, where
 $\mathcal{V}(G)$ is the number of minimal vertex covers of $G$ of minimal size.
 \item $e\left(\dfrac{S}{J(G)^{(s)}}\right)=\binom{s+1}{2}|E(G)|.$ 
\end{enumerate}	
\end{corollary}
\vskip 5mm \noindent
\textbf{Acknowledgement:} 
We would like to thank A. V. Jayanthan and J. K. Verma
for clarifications on several doubts.
The computational commutative algebra package Macaulay 2 \cite{M2} was heavily used to compute several
examples. 
The first named author is partially supported by NBHM, India.
The third named author is partially supported by UGC, India.
The last named author is partially supported by 
the Institute of Mathematical Sciences, Chennai and National Postdoctoral Fellowship (PDF/2019/002800) by Sciences and Engineering Research Board,  India.
We also thank the referee for carefully reading the manuscript and making
several suggestions that improved the exposition.

\bibliographystyle{abbrv}  
\bibliography{refs_reg}

\begin{thebibliography}{10}

\bibitem{selvi_ha}
S.~Beyarslan, H.~T. H{\`a}, and T.~N. Trung.
\newblock Regularity of powers of forests and cycles.
\newblock {\em J. Algebraic Combin.}, 42(4):1077--1095, 2015.

\bibitem{Bh1993}
W.~Bruns and J.~Herzog.
\newblock {\em Cohen-{M}acaulay rings}, volume~39 of {\em Cambridge Studies in
  Advanced Mathematics}.
\newblock Cambridge University Press, Cambridge, 1993.

\bibitem{CHT}
S.~D. Cutkosky, J.~Herzog, and N.~V. Trung.
\newblock Asymptotic behaviour of the {C}astelnuovo-{M}umford regularity.
\newblock {\em Compositio Math.}, 118(3):243--261, 1999.

\bibitem{DDAGHN}
H.~Dao, A.~De~Stefani, E.~Grifo, C.~Huneke, and L.~N\'{u}\~{n}ez Betancourt.
\newblock Symbolic powers of ideals.
\newblock In {\em Singularities and foliations. geometry, topology and
  applications}, volume 222 of {\em Springer Proc. Math. Stat.}, pages
  387--432. Springer, Cham, 2018.

\bibitem{DrabGue}
B.~Drabkin and L.~Guerrieri.
\newblock Asymptotic invariants of ideals with {N}oetherian symbolic {R}ees
  algebra and applications to cover ideals.
\newblock {\em J. Pure Appl. Algebra}, 224(1):300--319, 2020.

\bibitem{FHM}
C.~A. Francisco, H.~T. H\`a, and J.~Mermin.
\newblock Powers of square-free monomial ideals and combinatorics.
\newblock In {\em Commutative algebra}, pages 373--392. Springer, New York,
  2013.

\bibitem{GRV05}
I.~Gitler, E.~Reyes, and R.~H. Villarreal.
\newblock Blowup algebras of ideals of vertex covers of bipartite graphs.
\newblock In {\em Algebraic structures and their representations}, volume 376
  of {\em Contemp. Math.}, pages 273--279. Amer. Math. Soc., Providence, RI,
  2005.

\bibitem{GKP-crown}
M.~Glen, S.~Kitaev, and A.~Pyatkin.
\newblock On the representation number of a crown graph.
\newblock {\em Discrete Appl. Math.}, 244:89--93, 2018.

\bibitem{Good2013}
A.~Goodarzi.
\newblock On the {H}ilbert series of monomial ideals.
\newblock {\em J. Combin. Theory Ser. A}, 120(2):315--317, 2013.

\bibitem{M2}
D.~R. Grayson and M.~E. Stillman.
\newblock Macaulay2, a software system for research in algebraic geometry.
\newblock Available at \url{http://www.math.uiuc.edu/Macaulay2/}.

\bibitem{Gu17}
Y.~Gu.
\newblock Regularity of powers of edge ideals of some graphs.
\newblock {\em Acta Math. Vietnam.}, 42(3):445--454, 2017.

\bibitem{HT19}
H.~T. H\`a and N.~V. Trung.
\newblock Membership criteria and containments of powers of monomial ideals.
\newblock {\em Acta Math. Vietnam.}, 44(1):117--139, 2019.

\bibitem{HTT}
H.~T. H{\`a}, N.~V. Trung, and T.~N. Trung.
\newblock Depth and regularity of powers of sums of ideals.
\newblock {\em Math. Z.}, 282(3-4):819--838, 2016.

\bibitem{HTrung}
N.~T. Hang and T.~N. Trung.
\newblock Regularity of powers of cover ideals of unimodular hypergraphs.
\newblock {\em J. Algebra}, 513:159--176, 2018.

\bibitem{Herzog'sBook}
J.~Herzog and T.~Hibi.
\newblock {\em Monomial ideals}, volume 260 of {\em Graduate Texts in
  Mathematics}.
\newblock Springer-Verlag London, Ltd., London, 2011.

\bibitem{HHT07}
J.~Herzog, T.~Hibi, and N.~V. Trung.
\newblock Symbolic powers of monomial ideals and vertex cover algebras.
\newblock {\em Adv. Math.}, 210(1):304--322, 2007.

\bibitem{HHT02}
J.~Herzog, L.~T. Hoa, and N.~V. Trung.
\newblock Asymptotic linear bounds for the {C}astelnuovo-{M}umford regularity.
\newblock {\em Trans. Amer. Math. Soc.}, 354(5):1793--1809, 2002.

\bibitem{HT10}
L.~T. Hoa and T.~N. Trung.
\newblock Partial {C}astelnuovo-{M}umford regularities of sums and
  intersections of powers of monomial ideals.
\newblock {\em Math. Proc. Cambridge Philos. Soc.}, 149(2):229--246, 2010.

\bibitem{PhDT103J}
S.~Jacques.
\newblock {\em Betti numbers of graph ideals}.
\newblock PhD thesis, University of Sheffield, 2004.

\bibitem{JayanNeeraj}
A.~V. {Jayanthan} and N.~{Kumar}.
\newblock {Syzygies, Betti Numbers, and regularity of cover ideals of certain
  multipartite graphs}.
\newblock {\em Mathematics}, 869(7), 2019.

\bibitem{jayanthan}
A.~V. Jayanthan, N.~Narayanan, and S.~Selvaraja.
\newblock Regularity of powers of bipartite graphs.
\newblock {\em J. Algebraic Combin.}, 47(1):17--38, 2018.

\bibitem{vijay}
V.~Kodiyalam.
\newblock Asymptotic behaviour of {C}astelnuovo-{M}umford regularity.
\newblock {\em Proc. Amer. Math. Soc.}, 128(2):407--411, 2000.

\bibitem{RajivAjay}
A.~{Kumar} and R.~{Kumar}.
\newblock {Regularity, Rees algebra and Betti numbers of certain cover ideals}.
\newblock {\em In Prepartion}.

\bibitem{RajivAjay18}
R.~Kumar and A.~Kumar.
\newblock Certain classes of {C}ohen-{M}acaulay multipartite graphs.
\newblock {\em Comm. Algebra}, 47(5):1930--1938, 2019.

\bibitem{MMV19}
J.~Mart\'{\i}nez-Bernal, S.~Morey, R.~H. Villarreal, and C.~E. Vivares.
\newblock Depth and regularity of monomial ideals via polarization and
  combinatorial optimization.
\newblock {\em Acta Math. Vietnam.}, 44(1):243--268, 2019.

\bibitem{MT17}
N.~C. Minh and T.~N. Trung.
\newblock Regularity of symbolic powers and arboricity of matroids.
\newblock {\em Forum Math.}, 31(2):465--477, 2019.

\bibitem{NV14}
J.~Neves and M.~Vaz~Pinto.
\newblock Vanishing ideals over complete multipartite graphs.
\newblock {\em J. Pure Appl. Algebra}, 218(6):1084--1094, 2014.

\bibitem{Rather2018}
S.~A. Rather and P.~Singh.
\newblock On betti numbers of edge ideals of crown graphs.
\newblock {\em Beitr{\"a}ge zur Algebra und Geometrie / Contributions to
  Algebra and Geometry}, 2018.

\bibitem{Seyed16}
S.~A. Seyed~Fakhari.
\newblock Depth, {S}tanley depth, and regularity of ideals associated to
  graphs.
\newblock {\em Arch. Math. (Basel)}, 107(5):461--471, 2016.

\bibitem{Fakhari17}
S.~A. Seyed~Fakhari.
\newblock Depth and {S}tanley depth of symbolic powers of cover ideals of
  graphs.
\newblock {\em J. Algebra}, 492:402--413, 2017.

\bibitem{Seyed}
S.~A. Seyed~Fakhari.
\newblock Regularity of symbolic powers of cover ideals of graphs.
\newblock {\em Collect. Math.}, 70(2):187--195, 2019.

\bibitem{TW}
N.~V. Trung and H.-J. Wang.
\newblock On the asymptotic linearity of {C}astelnuovo-{M}umford regularity.
\newblock {\em J. Pure Appl. Algebra}, 201(1-3):42--48, 2005.

\bibitem{villarreal_book}
R.~H. Villarreal.
\newblock {\em Monomial algebras}.
\newblock Monographs and Research Notes in Mathematics. CRC Press, Boca Raton,
  FL, second edition, 2015.

\end{thebibliography}
\end{document}